\documentclass[12pt,reqno]{amsart}
\usepackage[letterpaper,top=2cm,bottom=2cm,left=3cm,right=3cm,marginparwidth=1.75cm]{geometry}
\usepackage{amsthm,amsmath,amsfonts,amssymb}
\usepackage{hyperref}
\usepackage{color}
\usepackage{graphicx}
\usepackage{color}
\usepackage{lipsum}

\newtheorem{theorem}{Theorem}

\newtheorem{algorithm}[theorem]{Algorithm}

\newtheorem{corollary}[theorem]{Corollary}

\newtheorem{definition}[theorem]{Definition}
\newtheorem{example}[theorem]{Example}

\newtheorem{lemma}[theorem]{Lemma}

\newtheorem{Question}[theorem]{Question}

\newtheorem{remark}[theorem]{Remark}

\def\rr{\mathbb{R}}

\def\rr{\mathbb{R}}

\newcommand{\R}{\mathbb{R}}

\begin{document}

\title[Structural stability of invasion graphs]{Structural stability of invasion graphs \\ for Lotka--Volterra systems}

\date{}

\author[Almaraz]{Pablo Almaraz$^{1,2}$}
\address{$^1$ Departamento de Ecuaciones Diferenciales y An\'{a}lisis Num\'{e}rico, Universidad de Sevilla, Campus Reina Mercedes, 41012, Sevilla, Spain.}
\address{$^2$Grupo de Oceanografía de Ecosistemas, Instituto de Ciencias Marinas de Andalucía (ICMAN-CSIC), Campus Universitario de Puerto Real, Puerto Real, 11519, Spain.}
\thanks{}

\author[Kalita]{Piotr Kalita$^{1,3}$}
\address{$^3$ Faculty of Mathematics and Computer Science, Jagiellonian University, ul. \L{}ojasiewicza 6, 30-348 Krak\'{o}w, Poland.}
\thanks{}

\author[Langa]{Jos\'e A. Langa$^1$}
\thanks{}

\author[Soler--Toscano]{Fernando Soler--Toscano$^4$}
\address{$^4$ Departamento de Filosofía, Lógica y Filosofía de la Ciencia, Universidad de Sevilla, C/ Camillo Jos\'{e} Cela, s/n, 41018, Sevilla, Spain.}
\thanks{}

\email[]{pablo.almaraz@csic.es} 
\email[]{piotr.kalita@ii.uj.edu.pl}
\email[]{langa@us.es}
\email[]{fsoler@us.es}

\begin{abstract}
In this paper, we study in detail the structure of the global attractor for the Lotka--Volterra system with a Volterra--Lyapunov stable structural matrix. We consider the invasion graph as recently introduced in \cite{Hofbauer} and prove that its edges represent all the heteroclinic connections between the equilibria of the system. We also study the stability of this structure with respect to the perturbation of the problem parameters. This allows us to introduce a definition of structural stability in ecology in coherence with the classical mathematical concept where there exists a detailed geometrical structure, robust under perturbation, that governs the transient and asymptotic dynamics.
\end{abstract}

\maketitle

\section{Introduction. Invasion graphs and ecological assembly}

The relations between populations of interacting species in ecosystems can be described by structured networks, where nodes represent species, and the edges represent the fact that the presence of one species affects another one \cite{Bascompte2014}. In order to understand the behavior of the ecosystem, however, it is necessary to study the dynamics of the interactions between species, i.e. how their quantities vary in time in relation to each other. The classical study of ecological dynamical models has been focused in their asymptotic behavior \cite{May1973}, but what is actually observed during the evolution in time of real systems is the presence of transient states \cite{Hastings2018}. These transient states are known to last for hundreds of generations in many natural systems in which stochasticity is an integral part of their dynamics \cite{Hastings2021}, so a major goal in current theoretical ecology is to evaluate the impact of transient dynamics on the persistence of communities in a constantly changing environment \cite{Hastings2018}. The interplay between transient and asymptotic dynamics is particularly important when we want to analyze the way in which communities assemble, or the invasion of one or several species to a given state of the ecosystem. The description of the sequence of both the potential assemblies or invasions (bottom-up), and the disassemblies or extinctions (top-down) is usually described by a network structure whose nodes are subcommunities and edges represent the possibility of evolution from one subcommunity to another \cite{Hang1993}. Full knowledge of such structure allows us to draw a complete landscape of all possible states in all possible times for the associated ecosystem. As it encompasses the essential information on the ecosystem, following our earlier terminology, we call it the informational structure (IS), cf. \cite{esteban, portillo}. The IS is the key object to investigate for a deeper understanding of the dynamics of the system, as it encapsulates both the transient states and the asymptotic dynamics. The complete characterization of the IS gives the information on the mechanics of ecological assembly. Indeed, given the close connection between the IS and the concept of assembly or community transition graph traditionally used in ecology \cite{Hang1993, Morton1996, servan}, the IS gives a picture of the pattern of possible developments of the ecological community containing the species present in the ecosystem.

If a model is a dissipative system of autonomous Ordinary Differential Equations which has a finite number of equilibria, then the underlying IS is contained in the structure of its global attractor. The nodes of IS correspond to the equilibria of the system and the edges represent the heteroclinic connections between them. In this paper, we focus on the Lotka--Volterra system of ODEs. While we choose this relatively simple model, there may exist many other factors affecting the evolution of an ecosystem, so that the modelling approach can include a variety of functional forms, from very basic ones to highly nonlinear vector fields including stochastic delays, or even higher order terms. The system under consideration here has the form
$$
u_i' = u_i\left(b_i+\sum_{j=1}^n a_{ij}u_j\right)\ \ \textrm{for}\ \ i\in \{ 1,\ldots,n\}.
$$
where $u_i$ is the state variable for species $i$ (e.g., population density or number of individuals); $b_i$ is the intrinsic growth rate for species $i$; and $a_{ij}$ is the direct effect of the average species $j$ individual on species $i$’s population growth rate \cite{novak}. We assume that the matrix $A = (a_{ij})_{i,j=1}^n$ is Volterra--Lyapunov stable (see Definition \ref{def1}).

For such a system, based on recent discoveries by Hofbauer and Schreiber \cite{Hofbauer}, we present an algorithm to construct the graph that represents the connections between the equilibria of the system, the IS, and we show that it is equivalent to the Invasion Graph (IG) as proposed in \cite{Hofbauer}. Thus, we complement the results of \cite{Hofbauer} which states, in a more general framework, that the graph of connections is a subgraph of the IG, but the possibility that IG is essentially bigger is generally not excluded. We show that for a particular case of Lotka--Volterra system with Volterra--Lyapunov stable matrix the two structures coincide. In this way, we give a joint framework for the study of ecological assembly \cite{servan}, Invasion Graphs \cite{Hofbauer} and Information Structures \cite{esteban, portillo}.

We stress that our argument works only in the Volterra--Lyapunov stable case where the IG (and equivalently IS) is the directed  graph and the results of \cite{Takeuchi} allow to construct the unique minimal element, the globally asymptotically stable stationary point (GASS). While this assumption may be restrictive, the advantage is that we explicitly describe the structure of all connections between the equilibria. In the general case, the Lotka--Volterra systems may encompass many rich dynamic phenomena, such as limit cycles \cite{Afraimovich2008}, but the analytical algorithm to construct the whole dynamics for a general $n$-dimensional system is still unknown. 

The problem of structural stability is of a fundamental importance in biology: it concerns the question of whether the state of a system and its stability will survive upon a small perturbation of model parameters. Recently, Rohr et al. \cite{rohr} represented the structural stability of ecological networks as a problem of community persistence. Essentially, the aim is to provide a measure of the range of admissible perturbations to a system under which no interacting species become extinct, i.e. the community is feasible. Feasibility refers here to the existence of a saturated equilibrium vector, that is, given a particular combination of species interaction parameters and intrinsic growth rates ($a_{ij}$ and $b_i$ in \eqref{lv}, respectively) all of the abundances are strictly positive at the equilibrium. Thus, there is a connection between structural stability, as it is currently used in ecology, and the Modern Coexistence Theory (MCT) \cite{barabas}, which aims at determining the number of species that can coexist in an ecosystem \cite{barabas}. Invasion Graphs, as introduced by \cite{Hofbauer}, extends the concept of assembly graphs to the invasibility criteria of the MCT: the condition that a set of persisting interacting species should have positive per-capita growths rates when rare \cite{chesson, barabas}. A novel contribution of our paper is to provide a link between Information Structures and Invasion Graphs through a measure of structural stability of global attractors that integrate both the transient and asymptotic dynamics. This achievement can be of paramount importance for a more detailed understanding of community coexistence and functioning in variable environments.

Inspired by these considerations, and by the study on the stability of the global attractor structure \cite{Bortolan}, we show that, not only the stable equilibrium but also the whole assembly remains unchanged upon a small perturbation of model parameters. This result is interesting from a mathematical point of view as we get a result on structural stability for a problem which is not necessarily Morse--Smale, contrary to many classical structural stability results (see \cite{Bortolan} and references therein). On the other hand, its interest from the point of view of ecology is that it links the concept of structural stability from \cite{rohr} with ecological assembly \cite{servan} and invasion dynamics \cite{Hofbauer}. Indeed, the notion of stability of all the assembly can be viewed as the refinement of the notion of the stability of the persistent equilibrium (see \cite{rohr}), as it induces the decomposition of the stability cones for the latter case into the smaller cones of the stability of assemblies. 

The structure of the paper is as follows: in Section \ref{sec:2} we formulate the problem and summarize its basic properties; in particular, we recall the result of \cite{Takeuchi} on the existence and characterization of a globally asymptotically stable steady state. The next Section \ref{sec:3} is devoted to local properties of the system: we explicitly linearize it around the equilibria and study the properties of this linearization. The first main result of the paper, which states that the IS coincides with the IG is contained in Section \ref{sec:4}. The following Section \ref{sec:5} contains the second main result, on the problem of structural stability, and on the stability cones for the assembly. Finally, in the appendices, we show that the considered problem is not necessarily Morse--Smale, and we formulate the open questions for the cases which are not  Volterra--Lyapunov stable.

\section{Lotka--Volterra systems and their global attractors.}\label{sec:2}
In this section we introduce the Lotka--Volterra systems and, for Volterra--Lyapunov stable matrices in the governing equation, we formulate the results on the underlying dynamics. The key concept is the global attractor. This attractor contains the minimal invariant sets (in our case, the equilibria) and the complete trajectories joining them in a hierarchical way. In our case, each admissible equilibrium or stationary point describes a subcommunity of the system. If this admissible equilibrium has strictly positive components, it is also feasible. Equilibria are joined by complete trajectories, i.e., global solutions of the system defined for all $t\in \R.$  This structure encodes all possible stationary states of the system and the underlying backward and forward behavior of the dynamics via the heteroclinic connections. It is a directed graph, which has been defined as an information structure in \cite{esteban,  kalita, portillo}, and it induces a landscape of the phase space defined as an informational field \cite{kalita}.

We start from definitions of classes of stable matrices. More information on them, as well as on the underlying dynamics of the associated Lotka--Volterra systems can be found in \cite{Hofbauer_Sigmund_1988, Hofbauer_Sigmund_1998, Logofet_1993, Takeuchi}.

\begin{definition}
	A real matrix $A\in \mathbb{R}^{n\times n}$ is stable if $\sigma(A)\subset \{ \lambda \in \mathbb{C}\, :\ \textrm{Re}\, \lambda < 0 \}$, where $\sigma(A)$ is the spectrum of $A$. 
\end{definition}

\begin{definition}
	A real matrix $A\in \mathbb{R}^{n\times n}$ is D-stable if for every matrix $D = \textrm{diag}\{ d_1,\ldots, d_n\}$  with $d_i>0$ for every $i$ the matrix $DA$ is stable. 
\end{definition}

\begin{definition}\label{def1}
	A real matrix $A\in \mathbb{R}^{n\times n}$ is Volterra--Lyapunov  stable (VL-stable) if there exists a matrix $H = \textrm{diag}\{ h_1,\ldots, h_n\}$ with $h_i>0$ such that $HA+A^TH$ is negative definite (i.e. stable).
\end{definition}

Consider the following Lotka--Volterra system with Volterra--Lyapunov stable matrix $A = (a_{ij})_{i,j=1}^n$ and a vector $b \in \rr^n$.
\begin{equation}\label{lv}
u_i' = F_i(u) = u_i\left(b_i+\sum_{j=1}^n a_{ij}u_j\right)\ \ \textrm{for}\ \ i\in \{ 1,\ldots,n\}.
\end{equation}

Let $n \in \mathbb{N}$. We denote
$$
\overline{C}_+ = \{ x = (x_1,\ldots,x_n)\in \mathbb{R}^n\,:\ x_i \geq  0 \ \ \textrm{for}\ \ i\in \{1,\ldots,n\}\},
$$
and 
$$
C_+ = \textrm{int}\ \overline{C}_+ = \{ x = (x_1,\ldots,x_n)\in \mathbb{R}^n\,:\ x_i >  0 \ \ \textrm{for}\ \ i\in \{1,\ldots,n\}\}.
$$
Now let $x = (x_1,\ldots,x_n) \in \overline{C}_+$. If $J\subset \{ 1,\ldots,n\}$ is a set of indices then we will use a notation
$$
C_+^J = \{ x\in \overline{C}_+\ :\  x_i > 0\ \ \textrm{for}\ \ i\in J  \}. 
$$
If $x\in \overline{C}_+$, then we denote $J(x) = \{ i\in \{1,\ldots,n\}\,:\ x_i > 0 \}$. Having such $x\in \overline{C}_+$, we have
$$
C_+^{J(x)} = \{ y\in \overline{C}_+\ :\  y_i > 0\ \ \textrm{for}\ \ i\in J(x)  \}.
$$

We present a result on the  system \eqref{lv} from \cite{Takeuchi}. We will be first interested in its equilibria in $\overline{C}_+$. Clearly $0=(0,\ldots,0)\in \rr^n$ is one of them. If we choose the nonempty subset of indices $J \subset \{1,\ldots,n\}$, say $J = \{ i_1,\ldots, i_m\}$, then we will say that this set defines an admissible equilibrium if there exists a point $x \in C^J_+$ with $x_i = 0$ for $i\not\in J$ which is an equilibrium of \eqref{lv}. The statement will be made more precise with some auxiliary notation introduced with the next definition.

\begin{definition}
	Let $A = (a_{ij})_{i,j=1}^n \in \rr^{n\times n}$ be a matrix. If $m<n$, then $m\times m$ principal submatrix of $A$ is obtained by removing any $n-m$ columns and $n-m$ rows with the same indices from $A$, i.e. if $1 \leq  i_1 < i_2 < \ldots < i_m \leq n$. The principal submatrix of $A$ associated with the set $J =\{i_1,\ldots,i_m\}$ has a form $A(J) = (A(J)_{jk})_{j,k=1}^m = (a_{i_ji_k})_{j,k=1}^m$. 
\end{definition} 

Also, for a vector $b\in \rr^n$ we can associate with a set of indices $J=\{i_1,\ldots,i_m\}$ its subvector $b(J) = (b_{i_j})_{j=1}^m$. So, the set $J$ defines a feasible equilibrium of the subsystem consisting only of the equations indexed by elements of $J$ and taking the variables outside $J$ as zero, if the solution of the system $A(J) v = -b(J)$ has all coordinates strictly positive. We denote this solution by $u^*(J)$. 
The associated admissible equilibrium of the original $n$-dimensional system is given by $u_i = 0$ for $i\not\in J$, and $u_{i_j} = v_j$ for $j\in \{1,\ldots,m\}$, i.e., $i_j\in J$. We use the notation 
$u^*=(u^*(J),0_{i\in \{1,\ldots,n\}\setminus J})$. 

Since every subset of $\{ 1,\ldots,n\}$ can potentially define an admissible equilibrium, there may be maximally $2^n$ of them (including zero), each of them determined uniquely by splitting $\{ 1,\ldots,n\}$ into the union of two disjoint subsets: the set $J$ on which the coordinates are strictly positive (this set defines the equilibrium) and the remainder on which they must be zero. 

It is not difficult to prove that for every subset of indices $J \subset \{1,\ldots, n\}$ the set
$$
\{ x\in \overline{C}_+\ :\ x_i = 0\ \ \textrm{for some}\ \ i\in J \} = \overline{C}_+\setminus C_+^J
$$
is positively and negatively invariant with respect to the flow defined by \eqref{lv}. 

We recall the definition of the Linear Complementarity Problem (LCP).
Given a matrix $B\in \mathbb{R}^{n\times n}$ and a vector $c\in \mathbb{R}^n$ the linear complementarity problem $LCP(B,c)$ consists in finding a vector $x\in \mathbb{R}^n$ such that
\begin{align*}
   	&	Bx+c\geq 0,\\
   	&	x\geq 0,\\
   	&	x^\top (Bx+c) = 0.
\end{align*}
If the matrix $A$ is Volterra--Lyapunov stable then the problem $LCP(-A,-b)$ has a unique solution for every $b\in \mathbb{R}^n$, cf. \cite[Lemma 3.2.1 and Lemma 3.2.2]{Takeuchi}.
%
    
    We cite the asymptotic stability result from \cite{Takeuchi}. 
    \begin{theorem}[\cite{Takeuchi}, Theorem 3.2.1]\label{thm:take}
    	If $A$ is Volterra--Lyapunov stable then for every $b\in \rr^n$ there exists a unique equilibrium $u^*\in \overline{C}_+$ of \eqref{lv} which is globally asymptotically stable in the sense that for every $u_0\in C_+^{J(u^*)}$ the solution $u(t)$ of \eqref{lv} with the initial data $u_0$ converges to $u^*$ as time tends to infinity. This $u^*$ is the unique solution of the linear complementarity problem $LCP(-A,-b)$. In particular, if the solution $\overline{u}$ of the system $A\overline{u}=-b$ is positive, then $u^* = \overline{u}$.   
    \end{theorem}
We will denote this $u^*$ as GASS (globally asymptotically stable stationary point). The following result is a straightforward consequence of the previous theorem.
\begin{corollary}
If $u^*$ is a GASS for the problem governed by \eqref{lv}, then for every set $J = \{j_1,\ldots,j_k\} \subset \{ 1,\ldots, n\}$, such that $J(u^*) \subset J$ the point $y\in \mathbb{R}^{k}$ defined by $y_i = u^*_{j_i}$ for $i\in\{1,\ldots,k\}$ is a GASS for the $k$ dimensional problem with $A(J)$ and $b(J)$. 
\end{corollary} 
We present the definition of a global attractor \cite{ha}:
\begin{definition}
	Let $X$ be a metric space and let $S(t):X\to X$ be a semigroup of mappings parameterized by $t\geq 0$. The set  $\mathcal{A} \subset X$ is called a global attractor for $\{S(t)\}_{t\geq 0}$ if it is nonempty, compact, invariant (i.e. $S(t)\mathcal{A} = \mathcal{A}$ for every $t\geq 0$), and it attracts all bounded sets of $X$ (i.e. if $B\subset X$ is nonempty and bounded then $\lim_{t\to \infty}\mathrm{dist}(S(t)B,\mathcal{A}) = 0$, where $\mathrm{dist}(C,D)=\sup_{x\in C}\inf_{y\in D}d(x,y)$ is the Hausdorff semidistance between sets $C,D\subset X$).
\end{definition}
If the mappings $S(t):X\to X$ are continuous, for the global attractor existence we need two properties to hold: the dissipativity and asymptotic compactness \cite{Robinson}. As  a consequence of Theorem \ref{thm:take} we have the following result.
\begin{theorem}
	For every $u_0\in \overline{C}_+$ the problem governed by \eqref{lv} has a unique solution which is a continuous function of time, and the initial data. Moreover, assuming the Volterra--Lyapunov stability of $A$, the problem has a global attractor.
\end{theorem}
\begin{proof}
	The result follows the argument of  \cite{LangaSuarez}. We only need to prove the dissipativity, i.e. the existence of the bounded absorbing set; once we have it, the asymptotic compactness is trivial. To this end it is sufficient to prove that if $\sum_{i=1}^n u_iw_i \geq R$ for $R$ large enough with some fixed weights $w_i>0$, then
	$$
	\frac{d}{dt} \sum_{i=1}^n u_i w_i \leq - D(R). 
	$$
	Indeed defining $|u|$ as $\sum_{i=1}^n u_i w_i$,
	\begin{align*}
	& \frac{d}{dt}  |u| = \frac{d}{dt} \sum_{i=1}^n u_i w_i  = \sum_{i=1}^n u_i b_i w_i + \sum_{i,j=1}^n u_i a_{ij} w_i u_j = \sum_{i=1}^n u_i b_i w_i + \frac{1}{2}\sum_{i,j=1}^n u_i (a_{ij} w_i + a_{ji} w_j) u_j\\
	& \ \ \leq c |u| - d|u|^2.
	\end{align*}
	where $c > 0$ and $d > 0$ are some constants.
	Then if $|u| \geq \frac{2c}{d}$, then the  right-hand side of the last expression is decreasing as a function of $|u|$, and  
	 	\begin{align*}
	 & \frac{d}{dt}  |u| \leq \frac{2c^2}{d} - d\frac{4c^2}{d^2} = -\frac{2c^2}{d}, 
	 \end{align*}
	 which is enough for the global attractor existence.
\end{proof}


\section{Equilibria and the local dynamics}\label{sec:3}
While it is straightforward to find all the equilibria of \eqref{lv} (it suffices to solve $2^n$ linear systems and determine the ones whose solutions are strictly positive, see also \cite{Lischke2017} for an efficient algorithm), finding the connections between them is a harder task. Our aim here is to give an algorithm that can be used to find exactly for which equilibria there exist heteroclinic connections, i.e. the solutions which tend to one equilibrium when time goes to minus infinity and another equilibrium when time goes to plus infinity. Before we move on to the study of the dynamics, we focus in this section on the local behavior in the neighborhood of the equilibria.

\subsection{Linearization and its properties.} We construct the linearized system in the neighborhood of the equilibrium $u^*$ of \eqref{lv}. Let $u^*$ be an equilibrium and denote $v=u-u^*$. Then the system \eqref{lv} can be rewritten as
\begin{equation*}
	 v_i' = \sum_{j=1}^n \frac{\partial F_i(u^*)}{\partial u_j} v_j + G_i(v), 
\end{equation*}
where $G_i(v) = \sum_{j=1}^n a_{ij}v_jv_i$ is the quadratic remainder term. Assume that $u^*$ is  an equilibrium in which the variables are sorted in such a way that $u_i^*\neq 0$ for $i=1,\ldots,k$ and $u_i^*=0$ for $i=k+1,\ldots,n$. Then 
for $i=1,\ldots,k$ the equation of the above system is
$$
v_i' =  \sum_{j=1}^k v_j a_{ij} u_i^*+\sum_{j=k+1}^n v_j a_{ij} u_i^*+G_i(v),
$$
and, for $i=k+1,\ldots,n$,
$$
v_i' =   v_i\left(b_i+\sum_{j=1}^k a_{ij}u_j^*\right)+G_i(v),
$$
The linearized system has the following block diagonal form 
\begin{equation}\label{eqn:linearized}
w' = Bw = \begin{pmatrix}
	B^{11} & B^{12}\\
	0 & B^{22}
\end{pmatrix}w,
\end{equation}
where the matrix $B^{22}$ is diagonal and $B^{22}_{ii} = b_i+\sum_{j=1}^k a_{ij}u_j^*$, while $B^{11}_{ij} = a_{ij}u_i^*$, and $B^{12}_{ij} = a_{ij}u_i^*$. 

We will name the subsets $J \subset \{1,\ldots,n\}$ corresponding to the admissible equilibria as admissible communities, according to the next definition.  
\begin{definition}
	The set (community) $I\subset \{1,\ldots,n\}$ is admissible if there exists the nonnegative equilibrium $u^*=(u^*_1,\ldots,u^*_n)$ of \eqref{lv} with $u^*_i>0$ if and only if $i\in I$. The family of all admissible communities will be denoted by $\mathcal{E} \subset 2^{\{1,\ldots,n\}}$.  The corresponding set of equilibria is denoted by $E = \{u^0,\ldots, u^K\}$.  As $0\in \mathcal{E}$ we always denote $u^0=0.$   
\end{definition}
 Whenever we speak about the multiple equilibria we will denote them by upper indices such as $u^i, u^j$. On the other hand, lower indices will denote coordinates of vectors $u=(u_1,\ldots,u_n)$.

The following lemmas summarize the properties of the matrix of the linearized system. Note that similar observations were made in different context in \cite{Lischke2017}.

\begin{lemma}\label{lem:stab}
	Assume that the matrix $A$ of the system \eqref{lv} is Volterra--Lyapunov stable. Consider the admissible community $I$ and the corresponding equilibrium $u^*$. The system linearized around $u^*$ has the form \eqref{eqn:linearized}. The spectrum of the matrix $B^{11}$ is contained in the open half-plane with the negative real part, i.e. $\sigma(B^{11}) \subset \{  z\in \mathbb{C}\;:\ \mathrm{Re}\, z < 0\}$. 
\end{lemma}
\begin{proof}
	Denote $u^*=(u_1^*,\ldots,u_n^*)$. The matrix $\{a_{ij}\}_{i,j\in I}$ as a principal submatrix of $A$ is Volterra--Lyapunov stable, cf. \cite[Theorem 1 c]{cross}. Hence is is also $D$-stable by \cite[Lemma 3.2.1]{Takeuchi}. This means that the  product $\textrm{diag}((u_i^*)_{i\in I})(a_{ij})_{i,j\in I}$ is stable. This  product is exactly $B^{11}$.  
\end{proof}

We are in a position to formulate a result of the properties of the linearized system \eqref{eqn:linearized}. 

\begin{lemma}\label{lem:res}
	Let $A$ be Volterra--Lyapunov stable and let $u^*$ be an equilibrium with the admissible community $\{1,\ldots,k\}$. The spectrum of the matrix $B$, denoted by $\sigma(B)$ is given by $\sigma(B) = \sigma(B^{11}) \cup \Lambda$ where $\lambda\in \Lambda$ if and only if  $\lambda = B^{22}_{ii}$ for some $i\in \{k+1,\ldots, n\}$, and  $\sigma(B^{11}) \subset \{ \textrm{Re}\, \lambda < 0\}$. So, if for some $\lambda\in \sigma(B)$ we have $\textrm{Re}\, \lambda \geq 0$, then $\lambda$ is real and $\lambda = B^{22}_{ii}$ for some $i\in \{k+1,\ldots,n\}$. The  eigenvector associated with the eigenvalue $B^{22}_{ii}$ is given by $x=(x_1,\ldots,x_k,0,\ldots,0,1,0,\ldots,0)$, where $1$ is on the position $i$ and $(x_j)_{j=1}^k$ is some vector. 
\end{lemma}

\begin{proof}
	The assertion that $\sigma(B^{11}) \subset \{ \textrm{Re}\, \lambda < 0\}$ follows from Lemma \ref{lem:stab}.  
	
	Now we prove that 
	$x=(x_1,\ldots,x_k,0,\ldots,0,1,0,\ldots,0)$ is the eigenvector associated with eigenvalue $B^{22}_{ii}$. We need to have
	$$
	B^{11}\begin{pmatrix} x_1\\. \\ . \\ . \\ x_k\end{pmatrix} + B^{12} \begin{pmatrix} 0\\. \\ 1 \\ . \\ 0\end{pmatrix} = B^{22}_{ii} \begin{pmatrix} x_1\\. \\ . \\ . \\ x_k\end{pmatrix}.  
	$$
	Such $(x_1,\ldots,x_k)$ can be found because the matrix $B^{11} - B^{22}_{ii} I$ is invertible as $B^{22}_{ii} \geq 0$ and it cannot be the eigenvalue of the matrix $B^{11}$ as all its eigenvalues have negative real part. Moreover
	$$
	0\begin{pmatrix} x_1\\. \\ . \\ . \\ x_k\end{pmatrix} + B^{22} \begin{pmatrix} 0\\. \\ 1 \\ . \\ 0\end{pmatrix} = B^{22}_{ii} \begin{pmatrix} 0\\. \\ 1 \\ . \\ 0\end{pmatrix},  
	$$ 
	holds trivially. Note that the result is also valid if the eigenvalues of $B^{22}$ have multiplicity greater than one.
\end{proof}

\section{Invasion graphs and information structures}\label{sec:4}
The main aim of this section is to propose the algorithm to determine the network of connections between equilibria, i.e. the graph for which the equilibria of the system correspond to the nodes, and the edges correspond to the heteroclinic connections. More specifically, the  vertices are given by the set of admissible communities $\mathcal{E}$ corresponding to the equilibria $E = \{ u^0, u^1, \ldots, u^K \}$ and the edge between two communities $J(u^i) \mapsto J(u^j)$ exists if and only if there exists a solution $\gamma(t)$ which connects $u^i$ with $u^j$ i.e. $\lim_{t\to -\infty}\|\gamma(t) - u^i\| = 0$ and  $\lim_{t\to \infty}\|\gamma(t) - u^j\| = 0$. Such solutions are called the heteroclinic connections. We show in this section that, in the Volterra--Lyapunov stable case, if we assume that all equilibria of the system are hyperbolic, then this graph is exactly the same as the Invasion Graph (IG) as defined by Hofbauer and Schreiber in \cite{Hofbauer}.

\subsection{Invasion rates and invasion graphs.}
Let $I\in \mathcal{E}$, i.e. $I$ is an admissible community of \eqref{lv}. For every species $i \in\{1,\ldots,n\}$, following Chesson \cite{chesson}  we define the invasion rates $r_i(I)$ (see \cite{barabas} for the recent overview of Chesson coexistence theory in which the key role is played by the invasion rates).
\begin{definition}
	Let $I\in \mathcal{E}$  and let $u^*$ be the related equilibrium such that $u^*_i > 0$ for $i\in I$ and $u^*_i = 0$ for $i\notin I$. Then the invasion rate of the species $i$ of the community $I$ is defined as
	$$
	r_i(I) = b_i + \sum_{j\in I}a_{ij}u_j^*. 
	$$
	If $I=\emptyset$ then we use the convention $r_i(\emptyset) = b_i$. 
\end{definition}

We first observe that the invasion rates are always zero for $i\in I$, this is a counterpart of Lemma 1 from \cite{Hofbauer}.
\begin{remark}
	If $i\in I$ then $r_i(I) = 0$. This follows from the fact that $u^*=(u^*(I),0_{i\in\{1,\ldots,n\}\setminus I})$ is an equilibrium, whence $A(I)u^*(I) = -b(I)$.
\end{remark} 
The remaining invasion rates are the eigenvalues of the system linearization at the equilibrium $u^*$. 
\begin{remark}
	If $i\notin I$ then the entries $B^{22}_{ii}$ of the linearization matrix $B$    given in Lemma \ref{lem:res} by $B^{22}_{ii} = b_i + \sum_{j\in I}a_{ij}u_j^*$, are the invasion rates $r_i(I)$.  
\end{remark}

Following \cite{Hofbauer} we present the construction of the Invasion Graph (IG) together with the result that all heteroclinic connections between the equilibria correspond to some edges in this graph. The construction and results in \cite{Hofbauer} are very general: they do not need the minimal invariant sets to be equilibria only, and the case of more general structures is covered too (see \cite{May}, where an example of periodic solutions is given; for such case the invasion rate is defined for an ergodic measure supported by such solution).  We restrict the presentation in this section to the simpler situation where the minimal isolated invariant sets (and thus the supports of the ergodic invariant measures) are only the equilibria of the system. While this is guaranteed to be true in the  case of a Lyapunov--Volterra stable matrix, this assumption is hard to verify in the case of a general $A$. 

We revisit the algorithm for constructing the IG, presented in \cite{Hofbauer}:

\begin{algorithm}\label{IG} The Invasion Graph is constructed in two steps: the first step defines its vertexes, and the second one its edges.

\begin{itemize}
	\item[(Step 1)]  The set of vertexes of the graph is $\mathcal{E}$, i.e., the vertexes are given by all admissible communities.
	
	\item[(Step 2)] The graph contains the edge from $I$ to $J$ (we denote it by $I\to J$) if $I\neq J$, $r_i(I) > 0$ for every $i\in J\setminus I$, and $r_i(J)< 0$ for every $i\in I\setminus J$.	
\end{itemize}
\end{algorithm}

In the graphs that we construct we identify equilibria with the sets of their nonzero variables which define them uniquely.  Hence sometimes we will speak of edges between the equilibria $u^i\to u^j$ and sometimes, equivalently between the sets of natural numbers such as, for example, $I\to J$.
 
The key property of the IG obtained in \cite{Hofbauer} is contained in the next result, cf. \cite[Lemma 2]{Hofbauer}.

\begin{lemma}\label{hof}
	Let $A, b$ be such that  $r_i(I)\neq 0$ for every $I\in \mathcal{E}$ and for every $i\notin I$. Assume that $\gamma(t)$ is the solution of \eqref{lv} with $\lim_{t\to-\infty}\|\gamma(t)-u^j\|=0$ and $\lim_{t\to\infty}\|\gamma(t)-u^k\|=0$, where $u^j,u^k$ are two equilibria of the system. Then, in the invasion graph there exists the edge from $J(u^j)$ to $J(u^k)$.
	\end{lemma}
	
We define the graph of connections:
	
	\begin{definition}
		The set of vertices of the graph of connections is given by $\mathcal{E}$. The edge $J(u^j)\to J(u^k)$, where $J(u^j), J(u^k)\in \mathcal{E}$ exists in the graph of connections if and only if there exists the solution $\gamma$ of \eqref{lv} such that $\lim_{t\to -\infty} \gamma(t) = u^j$ and $\lim_{t\to \infty} \gamma(t) = u^k$.
	\end{definition} 
 
Finally, following \cite{Hofbauer} we define the Invasion Scheme as the table of the signs of the invasion rates, i.e. 
$$
\mathbb{IS}(i,I) = \textrm{sgn}\, r_i(I)\ \ \textrm{for}\ \ I\subset\mathcal{E}, i\in \{1,\ldots,n\}.
$$
If $i\in I$ then always $\mathbb{IS}(i,I) = 0$. If for some $i\notin I$ we have $\mathbb{IS}(i,I) = 0$ then the equilibrium associated with $I$ is nonhyperbolic. In other cases, we always have $\mathbb{IS}(i,I) = 1$ or $\mathbb{IS}(i,I) = -1$. This matrix is sufficient to construct the IG.\\

\subsection{Finding the connections between equilibria} Lemma \ref{hof} guarantees that the existence of the edge in the IG is the necessary condition for the existence of the connection between equilibria. That is, the graph of connections is the subgraph of the IG. This section is devoted to the proof that this necessary condition is also sufficient for the case of a Volterra--Lyapunov stable matrix $A$.

\begin{theorem}\label{thm:36}
		Let $A$ be a Volterra--Lyapunov stable matrix. Let $u^*$ be an admissible equilibrium which corresponds to the community $I \in \mathcal{E}$. If the set $J \supset I$ is such that for every $j\in J \setminus I$ we have $r_j(I) > 0$ then there exists a solution $\gamma$ of \eqref{lv} such that $\lim_{t\to -\infty}\gamma(t) = u^*$ and $\lim_{t\to \infty}\gamma(t)$ is a GASS for the community $J$.   
\end{theorem}

\begin{proof}
	It is enough to show that the unstable manifold of the point $u^*$ in the nonnegative cone intersects the interior of the cone associated with $J$, denoted by $C_+^J$. Then the result follows by Theorem \ref{thm:take}. 
	 For the equilibrium $u^*$, by Lemma \ref{lem:res} the local unstable space $E_u$ contains the vector $(y_1,\ldots,y_k, 1_{i\in J\setminus I})$, where the characteristic vector $1_{i\in J\setminus I}$ of $B^{22}$ has coordinates equal to $1$ if $i\in J\setminus I$ and $0$ otherwise. Now by the local unstable manifold theorem, cf. \cite[Theorem 1]{Kelley}, \cite[Theorem 3.2.1]{Guckenheimer}, the manifold $W^u_{loc}(u^*)$, contains points
	$$
	u^\varepsilon = u^* + \varepsilon (y_1,\ldots,y_k, 1_{i\in J\setminus I}) + \Phi (\varepsilon (y_1,\ldots,y_k, 1_{i\in J\setminus I})),
	$$ 
	where $(y_1,\ldots,y_k)$ are given vectors independent of $\varepsilon$, with $\varepsilon>0$ being a sufficiently small number, and $\Phi$ being a smooth function with $\Phi(0)=0$ and $D\Phi(0)=0$. By the Taylor theorem for $j\in J\setminus I$
	$$
	u^\varepsilon_j = \varepsilon + C\varepsilon^2,
	$$
	where $C$ depends on $(y_1,\ldots,y_k)$ and $|C|$ is bounded by a constant depending on the maximum norm of the Hessian of $\Phi$ on the set $U$ which is a neighborhood of zero. Hence, for sufficiently small $\varepsilon > 0$ the local unstable manifold of $u^*$ contains points with all entries in $J\setminus I$ positive. As $u^*$ is positive on coordinates associated with $I$, the proof is complete.
\end{proof}

The  above result justifies the following algorithm, and we refer to the constructed graph as the Information Structure (IS; for the definition of the Linear Complementary Problem, refer to Section \ref{sec:2}).

\begin{algorithm}[Construction of IS]\label{algo}
In Step 1  for each subcommunity $J$ of $\{ 1,\ldots,n\}$ (including the empty set and the full subcommunity) we construct its GASS. Any trajectory with the initial data having positive entries on the coordinates in $J$ and zeros on coordinates outside $J$ will converge to this GASS.  
\begin{itemize}
	\item[(Step 1)] For all $2^n$ subcomunities in $\{ 1,\ldots, n\}$ find their GASSes by solving $LCP(-A(J),-b(J))$ for every subset $J\subset \{ 1,\ldots, n\}$. For each GASS the procedure returns also the set of its nonzero coordinates. 
In this step we not only construct GASSes for all communities in $\{ 1,\ldots,n\}$, but also find the set $\mathcal{E}$ of all admissible communities.   

	\item[(Step 2)] For every $I\in \mathcal{E}$ denote by $u^*$ the associated equilibrium. Draw outgoing edges from $I$  according to the algorithm below.
	
	\begin{itemize}
		\item[(1)]
		For every $i \in  \{1,\ldots,n\} \setminus I$ calculate the invasion rate 
		$$
		r_i(I) = b_i + \sum_{j\in I} a_{ij}u_j^*.  
		$$
		Take $J$ as the set of those $i \in  \{1,\ldots,n\} \setminus I$ for which $r_i(I) > 0$, i.e.  those species which can successfully invade the equilibrium community $I$. 
		\item[(2)] For every set $K$ such that $I \subsetneq K \subseteq I\cup J $ draw an edge from $I$ to $GASS(K)$.
		\end{itemize}
\end{itemize}
\end{algorithm}

\begin{remark}
	Note that the concept of the IS, as defined by this algorithm, only applies to Volterra-Lyapunov stable systems. This is because it relies on the existence of the GASS, which is characterized as the solution to the Linear Complementarity Problem, cf. Theorem \ref{thm:take}. However, the concept of IS as the skeleton of a global attractor is more general. Indeed, it can be defined as a graph whose vertexes are isolated invariant sets (see \cite{Aragao} for the concept of generalized gradient systems, where the set of connections is between sets more than equilibria) and edges in the IS are the possible connections between them.
\end{remark}

The following example illustrates the algorithm of the IS construction.
\begin{example}\label{ex:2}
	Consider the following system with the Volterra--Lyapunov stable matrix.
	\begin{align*}
		& u_1'=u_1(1.8-u_1+0.24u_2+0.11u_3+0.2u_4),\\
		& u_2'=u_2(-0.45+0.16u_1-u_2+0.05u_3+0.22u_4),\\
		& u_3'=u_3(0.14u_1+0.1u_2-u_3+0.18u_4),\\
		& u_4'=u_4(-0.35+0.11u_1+0.19u_2+0.01u_3-u_4).\\
	\end{align*}
	The IG (which must coincide with the IS and the graph of connections) of the above system is depicted in Fig. \ref{fig:ex2}.
	
	\begin{figure}
		\centering
		\includegraphics[width=0.5\linewidth]{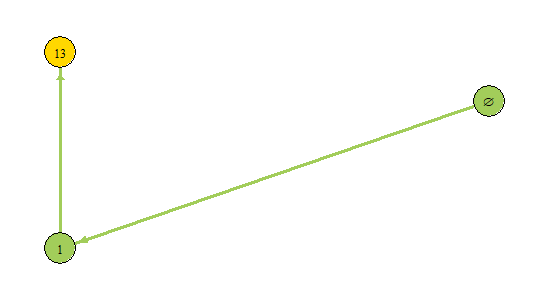}
		\caption{Information Structure for the problem given in Example \ref{ex:2}.}
		\label{fig:ex2}
	\end{figure}
	
	In order to run Algorithm \ref{algo} in the first step we solve the Linear Complementarity Problem for each subcommunity in order to find GASSes. The result is given in the following table:\\
	
	\begin{tabular}{|c|c||c|c||c|c||c|c|}
		\hline
		 Community & GASS & Community & GASS & Community  & GASS & Community & GASS \\
		\hline
		$\emptyset$ & $\emptyset$ & $\{4\}$ & $\emptyset$  & $\{2,3\}$ & $\emptyset$ & $\{1,2,4\}$  & $\{1\}$ \\
		\hline
		 $\{1\}$ & $\{1\}$ & $\{1,2\}$ & $\{1\}$ & $\{2,4\}$ & $\emptyset$ & $\{1,3,4\}$ &  $\{1,3\}$ \\
		\hline
		$\{2\}$ & $\emptyset$ & $\{1,3\}$ &  $\{1,3\}$& $\{3,4\}$ & $\emptyset$ & $\{2,3,4\}$ &  $\emptyset$\\
		\hline
		$\{3\}$ & $\emptyset$ & $\{1,4\}$ & $\{1\}$ & $\{1,2,3\}$ & $\{1,3\}$ & $\{1,2,3,4\}$  & $\{1,3\}$ \\
		\hline
	\end{tabular}
	
	\smallskip
	
Three equilibria were found in the course of computation of all GASSes, namely the equilibria corresponding to the communities $\emptyset, \{1\}, \{1,3\}$. For these communities, in Step 2 we first find those invasion rates by the species not belonging to them which are positive. These are: $r_1(\emptyset)$ and $r_3(\{1\})$. This means we have to draw edges from $\emptyset$ to $GASS(\{1\})$ and from $\{1\}$ to $GASS(\{1,3\})$. These are the two edges depicted in the graph.  
\end{example}

By Theorem  \ref{thm:36} we have the following Corollary
\begin{corollary}\label{corollary:incl}
	Let $A$ be Volterra--Lyapunov stable and let $u^j$ and $u^k$ be the admissible equilibria. If the above algorithm produces the edge from the community $J(u^j)$ to the community $J(u^k)$ then there exists the solution $
	\gamma$ of \eqref{lv} such that $\lim_{t\to -\infty}\gamma(t) = u^j$ and $\lim_{t\to \infty}\gamma(t) = u^k$.    
\end{corollary}

By the above corollary we can be sure that if the above algorithm produces the edge, then this edge represents the actual connection between the equilibria of the system. It is hence a kind of ``inner approximation'' of the graph of all connections between the equilibria. On the other hand, Lemma \ref{hof} implies that the IG of \cite{Hofbauer} is the ``outer approximation'', because every existing connection is represented in the IG. So, if we are able to prove that every connection present in the IG is also constructed by the above algorithm, we have the following chain of graphs, where each preceding graph is the subgraph of the next one:
	$$
	(\textrm{IS of Algorithm \ \ref{algo}}) \substack{(1)\\ \subset} (\textrm{Graph of connections}) \substack{(2)\\ \subset} (\text{Invasion Graph}) \substack{(3)\\ \subset} (\textrm{IS of Algorithm\  \ref{algo}}),
	$$
	and all three structures must coincide. 
The inclusion (1) follows from Corollary \ref{corollary:incl} and needs $A$ to be Volterra--Lyapunov stable. The inclusion (2) follows from Lemma \ref{hof}, and does not necessarily need the Volterra--Lyapunov stability. We continue by proving (3).

\begin{theorem}\label{thm:incl3}
	Assume the $A$ is Volterra--Lyapunov stable and that $u^i, u^k$ are the two admissible equilibria with the sets of corresponding nonzero coordinates given by $I_1 = J(u^i)$ and $I_2 = J(u^k)$. 
	Assume that the connection $I_1\mapsto I_2$ exists in IG, that is for every $j\in  I_2\setminus I_1$ we have $r_j(I_1)>0$ and for every $j\in I_1\setminus I_2$ we have $r_j(I_2) < 0$. Then the graph constructed by Algorithm \ref{algo} contains the edge $I_1 \mapsto I_2$.
\end{theorem}

\begin{proof}
	 Consider the system restricted to the variables in $I_1\cup I_2$, i.e. set $u_i=0$ for $i\notin I_1\cup I_2$. 
	 Clearly Algorithm \ref{algo} produces the edge from $J(u^i)=I_1$ to the community corresponding to the node $u^*$ which is the GASS for the community $I_1\cup I_2$. We need to prove that this GASS is $u^k$. Suppose that $u^k$ is not the GASS, i.e. $u^*\neq u^k$.  Then in the arbitrary neighbourhood of $u^k$ there exist points (in the interior of the cone $C_+^{I_1\cup I_2}$, strictly positive in the restricted variables) which are attracted to $u^*$. Since the matrix $B^{11}$ at the point $u^k$ is stable by Lemma \ref{lem:stab} and remaining eigenvalues (that of $B^{22}$) are given by $r_j(I_2)< 0$ for $j\in I_1\setminus I_2$ it follows that the spectrum of the Jacobi matrix at $u^k$ satisfies
	$$\sigma(B) \subset  \{\lambda \in \mathbb{C}\,:\ \textrm{Re}\, \lambda < 0 \}.$$
	 In particular, $B$ is hyperbolic and the local stable manifold of $u^k$ is the whole neighborhood of this point. But, since there exists a point in any neighborhood of $u^k$ attracted to $u^* \neq u^k$, the contradiction follows. 
\end{proof}

\begin{corollary}
	Assume that $A$ is Volterra--Lyapunov stable. Then, the IG is a subgraph of the graph of connections. If, additionally, all invasion rates $r_i(J)$ are nonzero for $i\notin J$ for all admissible communities $J\in \mathcal{E}$ (i.e. all equilibria corresponding to admissible communities are hyperbolic), then both graphs coincide.
\end{corollary}

Note that since Algorithm \ref{IG} does not need to find GASSes and solve LCPs, the construction of IG  is the way to find the graph of connections with the lower computational effort.

\begin{remark}
	We can summarize the obtained results as follows. 
	\begin{align*}
		& A\ \textrm{is Volterra--Lyapunov stable}\ \ \Rightarrow\\
		& \ \ \ \ (\text{Invasion Graph}) \substack{(3)\\ \subset} (\textrm{IS of Algorithm \ \ref{algo}}) \substack{(1)\\ \subset} (\textrm{Graph of connections}).\\ \\
		&
		r_i(J)\neq 0\ \textrm{for}\ i\notin J \textrm{(all equilibria are hyperbolic)}\ \Rightarrow\\
		& \ \ \ \ \  (\textrm{Graph of connections}) \substack{(2)\\ \subset} (\text{Invasion Graph}).
	\end{align*}
	In the proof of Theorem \ref{thm:incl3} we have also shown that if $A$ is Volterra--Lyapunov stable then the fact that $\sigma(DF(u^*))$ is hyperbolic (its spectrum does not intersect the imaginary axis) is equivalent to the statement that $r_i(J) \neq 0$ for every $i \notin J$. This fact follows from Lemma \ref{lem:stab}. Note that Theorems \ref{thm:36} and \ref{thm:incl3} remain valid even for nonhyperbolic case, i.e. if for some $j\notin I$ we have $r_j(I) = 0$ (in Theorem \ref{thm:36} we take only those $j\in K\setminus I$ for which $r_j(I) > 0$ so, in the nonhyperbolic case, if $r_j(I) = 0$, the species $j$ will not be considered as the one which may succesfully invade the community $I$, which may lead to omission of existing connections). Hence in the nonhyperbolic case the inclusions (1) and (3) remain valid, but not necessarily the inclusion (2). So without the hyperbolicity assumption, the IG is included in the graph of connections, but not necessarily the otherwise.  
\end{remark}

We remark that every graph that we construct must always represent a  substructure of the global attractor, since all equilibria and their heteroclinic connections belong to it. Moreover, if we assume that the global attractor consists only of the equilibria and their connections, then the constructed structure is exactly the global attractor, which is the case, for example, if the matrix $A$ is symmetric, cf. Section  \ref{symmetric}.  The following example demonstrates that this does not always has to be the case.
\begin{example}
	Consider the following system of three ODEs representing the May--Leonard problem, cf. \cite{chi, May}.
	\begin{align}
		&\nonumber u_1' = u_1(1-u_1-1.5 u_2-0.05u_3), \\
		& u_2'=u_2(1-0.05u_1-u_2-1.5u_3)\label{eq:may},\\
		&\nonumber u_3'=u_3(1-1.5 u_1-0.05u_2-u_3).
	\end{align}
	The matrix of the above system is Volterra--Lyapunov stable, cf. \cite{LVstable}. The system has five equilibria: zero, three one species equilibria and one three species equilibrium. The graph of connections (and, equivalently IG and IS) for the above system is presented in Fig. \ref{fig:may}. The graph does not represent the full dynamics of the system because inside the global attractor there exists a solution which converges forward in time to the three species equilibrium (represented by the node $123$) and backward in time to the heteroclinic cycle connecting the three nodes $1$, $2$, and $3$. 
	
	\begin{figure}
		\centering
		\includegraphics[width=0.5\linewidth]{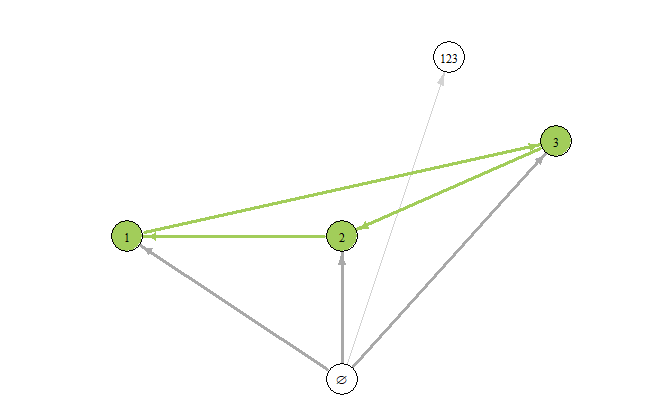}
		\caption{The Invasion Graph for the system \eqref{eq:may}}
		\label{fig:may}
	\end{figure}
	
\end{example}

Notably, in the above example the graph has a $3$-cycle consisting of three heteroclinic connections. We leave open the following question
\begin{Question}
	Assume that the matrix $A$ is Volterra--Lyapunov stable and that all equilibria are hyperbolic. Can we assert that if the IG is acyclic, then the global attractor consists only of the equilibria and their heteroclinic connections, and hence, this graph represents the whole dynamics of the system?
\end{Question}
%

\section{Structural stability of invasion graphs}\label{sec:5}

\subsection{Local structural stability}
If the system is Morse--Smale, then it is also structurally stable, i.e. $C^1$ small  perturbation of its vector field produces a system whose global attractor has the same structure (see \cite[Theorem 2]{Bortolan}). In this section we show that, although the system governed by \eqref{lv} is not necessarily Morse--Smale, cf. Example \ref{exmorse}, if all equilibria are hyperbolic, the small perturbation of $A$  and $b$  produces a system with the same graph of connections (and, if the global attractor consists only of the equilibria and their connections, with the same global attractor structure). In the next result $B(A,\varepsilon)$ denotes the euclidean ball in $\mathbb{R}^{n\times n}$ and $B(b,\varepsilon)$ in the euclidean ball in $\mathbb{R}^{n}$. Moreover, denote by $\mathcal{E}(A,b)$ the set of admissible communities for the problem with matrix $A$ and vector $b$. For $I \in \mathcal{E}(A,b)$ and $i\notin I$ we will use the notation $r_i^{A,b}(I)$ to denote the invasion rate corresponding to $A, b$.

\begin{theorem}\label{thm:stability}
	Let $\overline{A}$ be a Volterra--Lyapunov stable matrix and let $\overline{b}\in \mathbb{R}^n$ be such that for all admissible communities $I\in  \mathcal{E}(\overline{A},\overline{b})$ the corresponding equilibria are  hyperbolic.  Then there exists $\varepsilon > 0$ such that for all matrices $A\in B(\overline{A},\varepsilon)$    and all vectors $b\in B(\overline{b},\varepsilon)$ we have $\mathcal{E}(A,b)= \mathcal{E}(\overline{A},\overline{b})$. Moreover for every $I \in \mathcal{E}(\overline{A},\overline{b})$ and every $i\notin I$ we have
	\begin{align*}
	& r^{\overline{A},\overline{b}}_i(I) > 0 \Rightarrow r^{{A},{b}}_i(I) > 0,\\
	& r^{\overline{A},\overline{b}}_i(I) < 0 \Rightarrow r^{{A},{b}}_i(I) < 0.
	\end{align*}
	Hence, the edges in both Invasion Graphs for $\overline{A}, \overline{b}$ and $A, b$ are the same. This implies that  the graphs of connections for the problems with $\overline{A}, \overline{b}$ and ${A}, {b}$ coincide and the problem governed by the matrix $\overline{A}$ and the vector $\overline{b}$ is structurally stable in the class of Volterra--Lyapunov stable matrices.  
\end{theorem}

\begin{proof}
The fact that all equilibria are hyperbolic means that $r_i(I) \neq 0$ for every $i\notin I$ and every $I\in \mathcal{E}(\overline{A},\overline{b})$. 
	Note that since the eigenvalues depend continuously on the matrix, the set of Volterra--Lyapunov stable matrices is open and hence we can choose $\varepsilon$ such that every $A\in B(\overline{A},\varepsilon)$ is Volterra--Lyapunov stable. Now, the fact that $I\in \mathcal{E}(\overline{A},\overline{b})$ means that $\overline{A}(I)u^*(I) = -\overline{b}(I)$ has a positive solution $u^*(I)$. From that fact that the mapping $(A,b) \mapsto u^*$, which assigns to a nonsingular $k\times k$ matrix $A$ and vector $b\in \mathbb{R}^k$ the solution $u^*$ of the system $Au^* = -b$, which is continuous, we deduce that we can find $\varepsilon > 0$ such that all admissible communities remain admissible.
	
	We prove that a nonadmissible community for $(\overline{A},\overline{b})$ cannot produce an admissible one upon sufficiently small perturbation. Assume that $I \subset \{ 1,\ldots,n \}$ is not admissible. If at least one of the coordinates of the solution of the system $\overline{A}(I)u^*(I) = -\overline{b}(I)$ is negative, then this negativity is preserved upon small perturbation of $(\overline{A},\overline{b})$. If all are nonnegative, but at least one is zero, say $u_j^* = 0$, then
	$$
	\sum_{i\in I, i\neq j}\overline{a_{ki}}u_{i}^* + \overline{b_k} = 0\ \ \textrm{for every}\ \ k\in I.
	$$
	In particular
	$$
	\sum_{i\in I \setminus\{j\}}\overline{a_{ji}}u_{i}^* + \overline{b_j} = 0.
	$$
	Denote by $I_0\subset I$ the (possibly empty) set of coordinates for which entries of $u^*$ are positive. Then $I_0$ corresponds to the admissible community for $(\overline{A},\overline{b})$. The last equality means that $r^{\overline{A},\overline{b}}_j(I_0) = 0$, which contradicts the assumption of hyperbolicity. We have proved that $\mathcal{E}(A,b) = \mathcal{E}(\overline{A},\overline{b})$. 
	
	The invasion rates $r_i(I)$ are the continuous functions of the vector $b$, matrix $A$ and the equilibrium $u^*$ related to the admissible community $I$. This means that if $r_i(I)$ is nonzero for the system governed by $(\overline{A}, \overline{b})$, it remains nonzero, and does not change sign, in a small neighbourhood of $(\overline{A}, \overline{b})$. This completes the proof.
\end{proof}

\subsection{The regions of structural stability.} In this section we fix the matrix $A$ and we study the properties of the sets of vectors $b \in \mathbb{R}^n$  for which the Invasion Graphs remain unchanged.

\begin{theorem}\label{thm:max}
	Let ${A}$ be a Volterra--Lyapunov stable matrix and let $\overline{b}\in \mathbb{R}^n$ be such that for all admissible communities $I\in \mathcal{E}({A},\overline{b})$ the corresponding equilibria are hyperbolic.  Then there exists the unique maximal open neighbourhood $\mathcal{N}$ of $\overline{b}$ in $\mathbb{R}^n$ such that for every $b\in \mathcal{N}$ we have $\mathcal{E}(A,b) = \mathcal{E}(A,\overline{b})$, all admissible equilibria corresponding to $A,b$ are hyperbolic, and the Invasion Graphs coincide.  
\end{theorem}

\begin{proof}
	From Theorem \ref{thm:stability} we know that there exists an open neighborhood of $\overline{b}$ such that the properties required by the theorem are satisfied. Let us denote by $\mathfrak{N}(\overline{b})$ the family of all such neighbourhoods. It is nonempty. Then $\mathcal{N}$ is the union of all elements of $\mathfrak{N}(\overline{b})$.
\end{proof}

We continue by proving the lemma on convexity
\begin{lemma}\label{lemma:conv}
	Let $A$ be Volterra--Lyapunov stable. Suppose that $b_1, b_2\in \mathbb{R}^n$ are such that 
	\begin{itemize}
		\item $\mathcal{E}(A,b_1) = \mathcal{E}(A,b_2) = \mathcal{E}$,
		\item for every $I \in \mathcal{E}$ and for every $i \not\in I$ we have $r^{A,b_1}_{i}(I) \neq 0$, $r^{A,b_2}_{i}(I) \neq 0$, and $r^{A,b_1}_{i}(I) > 0 \Leftrightarrow r^{A,b_2}_{i}(I) > 0$.  
	\end{itemize}
	 Then for every $\lambda\in [0,1]$, denoting $b_\lambda = \lambda b_1 + (1-\lambda)b_2$, we have
	 \begin{itemize}
	 	\item $\mathcal{E}(A,b_\lambda) = \mathcal{E}$,
	 	\item for every $I \in \mathcal{E}$ and for every $i \not\in I$ we have $r^{A,b_\lambda}_{i}(I) \neq 0$, and $r^{A,b_1}_{i}(I) > 0 \Leftrightarrow r^{A,b_\lambda}_{i}(I) > 0$.  
	 \end{itemize} 
\end{lemma}
\begin{proof}
	Assume that $I\in \mathcal{E}$. Then  $A(I)u^*_1 = - b_1(I)$ and $A(I)u^*_2 = - b_2(I)$. This means that $A(I)(\lambda u^*_1 + (1-\lambda)u^*_2)= - b_\lambda(I)$ and $I$ is admissible for $b_\lambda$. On the other hand, assume that $I$ is admissible for $b_\lambda$, i.e. $I\in \mathcal{E}(A,b_\lambda)$ but $I\not\in \mathcal{E}$. Restrict the system to those unknowns which correspond to the indices of $I$. The fact that $I$ is admissible for the problem with $b_\lambda$ means that this system has the strictly positive equilibrium, i.e. the solution $v$ of $A(I)v = - b_\lambda(I)$ has all coordinates strictly positive, and by Theorem \ref{thm:take} it has to be the GASS of the system with $A(I), b_\lambda(I)$, and the solution of  $LCP(-A(I),-b_\lambda(I))$. On the other hand, as $I\not\in \mathcal{E}$, the solutions $u_1^*$ and $u_2^*$ of the problems $LCP(-A(I),-b_1(I))$ and $LCP(-A(I),-b_2(I))$, cannot have all coordinates strictly positive, and, because the invasion schemes $\mathbb{IS}$ for $b_1(I)$ and $b_2(I)$ coincide, the indices of zero and nonzero coordinates in both $u_1^*$ and $u_2^*$ are the same. Then $\lambda u_1^* + (1-\lambda)u_2^*$ must solve $LCP(-A(I),-b_\lambda(I))$ and hence it must be a GASS for the problem with $b_\lambda(I)$, a contradiction with the fact that this GASS has all coordinates strictly positive.         
	
	A straightforward calculation shows that for $I\in \mathcal{E}$
	$$
	r^{A,b_\lambda}_{i}(I) = \lambda r^{A,b_1}_{i}(I) + (1-\lambda)r^{A,b_2}_{i}(I),
	$$
	which is sufficient to complete the proof of the Lemma.
\end{proof}

\begin{remark}
	Note that, by Lemma \ref{lem:res} the assumption that all invasion rates $r_i(I)$ are nonzero for $i\notin I$ is equivalent to saying that the admissible equilibrium corresponding to $I$ is hyperbolic.  
\end{remark}

The next theorem states that the maximal neighbourhoods of Theorem \ref{thm:max} are convex cones and they group all points with a given invasion scheme $\mathbb{IS}$, i.e. the given configuration of equilibria and signs of invasion rates.

\begin{theorem}
	The maximal neighbourhood $\mathcal{N}$ of $\overline{b}$ given in Theorem \ref{thm:max} is an open and convex cone.  	Moreover, if for some point $b\in \mathbb{R}^n$ with all admissible equilibria being hyperbolic the invasion schemes for $A,b$ and $A,\overline{b}$ are the same, then $b\in \mathcal{N}$.  
\end{theorem}

\begin{proof}
	We first prove the second assertion. Take $b\in \mathbb{R}^n$ satisfying the assumptions of the theorem. By Lemma \ref{lemma:conv} the same assumptions are satisfied by every $b_\lambda \in \{ \lambda \overline{b}+(1-\lambda)b\,:\ \lambda\in[0,1]\}$. By Theorem \ref{thm:stability} for each $\lambda \in [0,1]$ there exists an open neighborhood of $b_\lambda$ on which the same assumptions also hold. The union of these neighborhoods is an open neighborhood of $\overline{b}$ which must be contained in $\mathcal{N}$ and contains $b$. 

	Now, convexity of $\mathcal{N}$ follows from Lemma \ref{lemma:conv}. To prove that $\mathcal{N}$ is a cone it is sufficient to see that 
$A(I)(u^*) = - b(I) \Rightarrow A(I)(\alpha u^*) = - \alpha b(I)$ and $r_i^{A,\alpha b}(I) = \alpha r_i^{A,b}(I)$.  
\end{proof}

As a consequence of the above results we can represent the space $\mathbb{R}^n$ as a union of finite number of disjoint open convex cones $\mathcal{N}_k$, with each cone corresponding to a given structure of the Invasion Graph, or equivalently, to a given $\mathbb{IS}$. This $\mathbb{IS}$ is the same for every $b$ in the cone. Note that the number of cones is finite as the number of possible invasion schemes is finite and any two vectors $b$ which yield the same scheme must belong to the same cone.  The points of nonhyperbolicity (that is, vectors $b$ where at least one of the admissible equilibria is nonhyperbolic) constitute the residual set $\mathcal{C}$.
$$
\mathbb{R}^n = \sum_{k=1}^L \mathcal{N}_k \cup \mathcal{C}. 
$$
The following statement holds
\begin{align*}
& b\in \mathcal{C} \Leftrightarrow \ \textrm{there exists}\ I\subset \{ 1,\ldots, n\}\ \textrm{and}\ i\notin I,\\
& \qquad \qquad \textrm{such that}\ r_i(I) = 0\ \textrm{and}\ u^*(I) > 0,\ \textrm{where}\ A(I)u^*(I) = -b(I).
\end{align*}
In other words, denoting by $(A(I)^{-1})_{ij} = a_{ij}(I)^{-1}$ the entries of the inverse matrix to $A(I)$.
\begin{align*}
& b\in \mathcal{C} \Leftrightarrow \ \textrm{there exists}\ I\subset \{ 1,\ldots, n\}\ \textrm{and}\ i\notin I,\textrm{such that}\\
& \qquad \qquad \ A(I)^{-1}b(I) < 0 \ \textrm{and}\ b_i -  \sum_{k\in I}\sum_{j\in I}a_{ij} a_{jk}(I)^{-1}b_k = 0.
\end{align*}
This means that
$$
\mathcal{C} \subset \bigcup_{I\subset \{1,\ldots,n\}} \bigcup_{i\in \{1,\ldots,n\}\setminus I} \left\{ b\in \mathbb{R}^n\,:\  b_i -  \sum_{k\in I}\sum_{j\in I}a_{ij} a_{jk}(I)^{-1}b_k = 0 \right\},
$$
i.e. the set of points of nonhyperbolicity is a subset of the union of a finite number of $n-1$ dimensional hyperspaces in $\mathbb{R}^n$. In particular, $\mathcal{C}$ is ``small'' compared to the sets $\mathcal{N}_k$. 

\section{Appendix A. The dynamical system generated by \eqref{lv} is not Morse--Smale.}
We begin this short section with a definition of a Morse--Smale system. We do not recall here all the necessary concepts: we refer, for example, to \cite[Section 2.1]{Bortolan} for details on all notions presented in this chapter. Note that related definition in \cite{Bortolan} is more general: it allows for existence of periodic orbits. We present its simplified version only for gradient-like systems.  
\begin{definition}
	Let $X$ be a Banach space and let $S(t):X\to X$ for $t\geq 0$ be a $C^1$ reversible semigroup with a global attractor $\mathcal{A}\subset X$. We denote the set of equilibria of $\{S(t)\}_{t\geq 0}$ as ${E}$, i.e. ${E} = \{ u\in X\,:\ S(t)u=u\}$ for every $t\geq 0$. The semigroup is Morse--Smale if 
	\begin{itemize}
		\item The global attractor consists of the equilibria  ${E}$, and nonconstant trajectories $\gamma:\mathbb{R}\to X$ such that $\lim_{t\to -\infty}\gamma(t) = u_1^*$ and $\lim_{t\to \infty}\gamma(t) = u_2^*$ where $u_1^*, u_2^* \in {E}$. 
		\item The set ${E}$ is finite and all equilibria in ${E}$  are hyperbolic.
		\item If $z\in \mathcal{A}$ is a nonequilibrium point such that $\lim_{t\to -\infty}S(t)z = u_1^*$ and $\lim_{t\to \infty}S(t)z = u_2^*$, then the unstable manifold of $u_1^*$ and stable manifold of $u_2^*$ intersect  transversally at every point $z$ of intersection, that is the sum of their tangent spaces at $z$ span the whole space $X$: $T_z(W^u(u_1^*)) + T_z(W^s(u_2^*)) = X$.   
	\end{itemize}
\end{definition}
The Lotka--Volterra system \eqref{lv} is defined on the closed positive cone $\overline{C}_+$. If a dynamical system is defined on a manifold $M$, then the transversality condition has the form $T_z(W^u(u_1^*)) + T_z(W^s(u_2^*)) = T_zM$, i.e. the tangent spaces of the stable and unstable manifolds span the tangent space of the whole $M$. It is possible to generalize the concept of Morse--Smale semigroups to manifolds with boundary, cf. \cite{LabarcaPacifico, PrishlyakBilunPrus, CRobinson}. However, instead of taking this path, we extend the system to the whole $\R^n$ and we provide a simple example that the resultant system is not necessarily Morse--Smale. Hence, while the structural stability results are known to hold for Morse--Smale systems, cf., for example, \cite{Bortolan}, our Theorem \ref{thm:stability} is a structural stability result (in positive cone) beyond this class.

\begin{example}\label{exmorse}
	\begin{figure}[h!]
		\centering
		\includegraphics[width=0.4\linewidth]{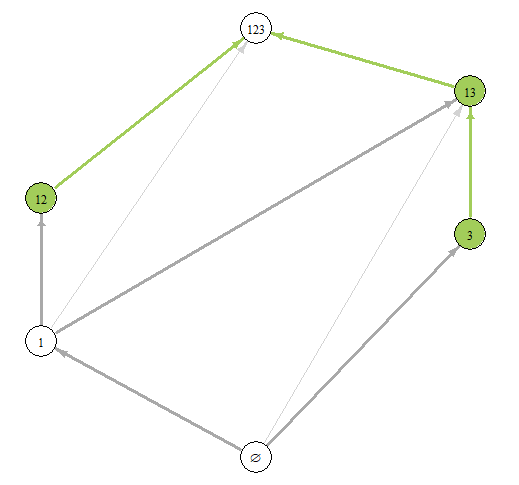}
		\caption{Invasion graph for the example given by \eqref{exmorse} for an ecological community with three species.  The GASS represents a feasible community (all species are present), $u^* = (0.2633778, 0.1695335, 0.377100)$. While the connections $\{3\}\to \{1,3\}$ and $\{1,3\}\to \{1,2,3\}$ are present in the graph, there is no connection $\{3\}\to \{1,2,3\}$}
		\label{fig:Fig_App}
	\end{figure}
	Consider the system
	\begin{equation}\label{eq:example}
		\begin{aligned}
			u_1' = u_1(-u_1 + 0.08u_2-0.47u_3 + 0.43),\\
			u_2'= u_2(0.66u_1-u_2+0.12u_3-0.05),\\
			u_3' = u_3(0.56u_1-0.28u_2-u_3+0.28).
		\end{aligned}
	\end{equation}
	Two of its admissible communities are $\{ 3 \}$ and $\{ 1,3\}$ with the corresponding equilibria $u_1^* = (0,0,0.28)$ and $u_2^*= (0.2362255,0,0.4122863)$. The matrix $A$, as it is diagonally dominant, is Volterra--Lyapunov stable. The graph of connections is presented in Fig. \ref{fig:Fig_App}, and contains the connection $\{3\}\to \{1,3\}$.  Both the unstable manifold of $\{3\}$ and the stable manifold of $\{1,3\}$ are contained in the $\{1,3\}$ plane.
	In fact $W^u(\{3\}) \subset W^s(\{1,3\})$, with $W^s(\{1,3\})$ being two-dimensional, and hence the sum of their tangent spaces cannot span the whole space.

	Also, note that if we set $u_2=0$ and restrict the system to $u_1, u_3$ variables only, for the resulting two dimensional system 
		\begin{equation}\label{eq:example2}
		\begin{aligned}
			u_1' = u_1(-u_1 -0.47u_3 + 0.43),\\
			u_3' = u_3(0.56u_1-u_3+0.28),
		\end{aligned}
	\end{equation}
the same intersection, which was non-transversal in the 3D problem, becomes transversal, and the system is now Morse--Smale. 
\end{example}

\section{Appendix B. Questions and open problems}
\subsection{Wider classes of stable matrices}The first question that we pose is  related with the fact that in \cite{rohr} the authors conjecture about stability of admissible equilibria for more general class of matrices $A$ - D-stable ones and stable ones. As our result on the graph of connections relies on Theorem \ref{thm:take} which uses the logarithmic Lyapunov function valid only in the class of Volterra--Lyapunov stable matrices $A$, it remains open to see if is holds in those wider classes.   
\begin{Question}
	Characterize the class of matrices for which the Invasion Graph (IG) corresponds to the graph of connections. Is this class essentially bigger then Volterra--Lyapunov stable ones? How does it relate to weaker notions of stable matrices such as, for instance, D-stable ones?
\end{Question}

\subsection{Symmetric case}\label{symmetric}

If the matrix $A$ is symmetric then the following function, as proposed by MacArthur \cite{Mac}, is Lyapunov:
$$
V(u) = -\sum_{i=1}^n b_i u_i  - \frac{1}{2}\sum_{i,j=1}^n a_{ij}u_i u_j.
$$
Indeed, after calculations we get,
\begin{equation}\label{lyapunov}
\frac{d (V(u(t)))}{dt} = V'(u) u'(t) = -\sum_{i=1}^n\left(b_i  + \sum_{j=1}^n a_{ij}u_j\right)^2 u_i.
\end{equation}

If we assume that $A$, together with all its principal minors are nonsingular, then the problem \eqref{lv} has the finite number of admissible equilibria which can all be explicitly calculated and one can construct the IG with the vertexes being exactly the equilibria. So for the case of symmetric and stable matrix $A$ (for symmetric matrices stability and Volterra--Lyapunov stability are the same), the IG exactly corresponds to the structure of the global attractor. Again, the open question appears.
\begin{Question}
	Does the IG correspond to the structure of a global attractor for a not necessarily  stable but symmetric matrix $A$ which, together with its all principal minors, is nonsingular? 
	\end{Question}

While we do not know how to answer this question we show that for symmetric case the Lyapunov function $V$ drops along every edge in the IG. While this fact does not guarantee the existence of the connection between the equilibria, this shows that the criterium associated with the Lyapunov function $V$ cannot exclude the edges in IG. 

\begin{lemma}
	Let $A$ be symmetric such that together with all its principal minors it is nonsingular and let $u^1$ and $u^2$ be admissible equilibria of \eqref{lv} which correspond to the communities $I_1, I_2$. If there exists an edge $I_1\to I_2$ in IG then $V(u^1) > V(u^2)$.  
\end{lemma}

\begin{proof}
We may assume without loss of generality that $I_1\cup I_2 = \mathbb{R}^n$. Otherwise we remove from the system the equations which correspond to the variables outside $I_1\cup I_2$. We represent  $\mathbb{R}^n = \mathbb{R}^{I_1\setminus I_2} \cup \mathbb{R}^{I_1\cap I_2}\cup \mathbb{R}^{I_2\setminus I_1}$, and we denote the projections on three subspaces as $\Pi_1, \Pi_2, \Pi_3$. Then, the matrix $A$ of the system can be written as
$$
A = \begin{pmatrix} B & C & D\\
	C^\top & E & F\\
	D^\top & F^\top & G
	\end{pmatrix}.
$$   
Now as $u_1$ is equilibrium related with $I_1$, hence $\Pi_3u^1 = 0$, $B\Pi_1u^1+C\Pi_2u^1 = -\Pi_1b$, and $C^\top \Pi_1u^1+E\Pi_2 u^1 = -\Pi_2 b$. Moreover, as the invasion rates at $u^1$ must be positive, it follows that $D^\top \Pi_1u^1 +F^\top\Pi_2u^1 > -\Pi_3 b$. 
Similar analysis at $u^2$ yields $\Pi_1u^2 = 0$, $E\Pi_2 u^2 + F\Pi_3 u^2 = -\Pi_2 b$, and $F^\top\Pi_2 u^2 + G\Pi_3 u^2 = -\Pi_3 b$. Finally as invasion rates at $u^2$ are negative we have  $C\Pi_2 u^2 + D \Pi_3 u^2 < -\Pi_1 b$. 
It follows that 
\begin{align*}
& (\Pi_3 u^2)^\top D^\top \Pi_1u^1 +(\Pi_3 u^2)^\top F^\top\Pi_2u^1 > -(\Pi_3 u^2)^\top \Pi_3 b,\\
& (\Pi_1 u^1)^\top C\Pi_2 u^2 + (\Pi_1 u^1)^\top D \Pi_3 u^2 < -(\Pi_1 u^1)^\top \Pi_1 b.
\end{align*}
Combining the two above inequalities we deduce
$$
(\Pi_1 u^1)^\top C\Pi_2 u^2 + (\Pi_1 u^1)^\top \Pi_1 b < (\Pi_3 u^2)^\top F^\top\Pi_2u^1 + (\Pi_3 u^2)^\top \Pi_3 b.
$$
But 
\begin{align*}
& (\Pi_2u^2)^\top C^\top \Pi_1u^1+(\Pi_2u^2)^\top E\Pi_2 u^1 = -(\Pi_2u^2)^\top \Pi_2 b,\\
& (\Pi_2u^1)^\top E\Pi_2 u^2 + (\Pi_2u^1)^\top F\Pi_3 u^2 = -(\Pi_2u^1)^\top \Pi_2 b.
\end{align*}
Hence
$$
-(\Pi_2u^2)^\top \Pi_2 b  - (\Pi_2u^2)^\top E\Pi_2 u^1 + (\Pi_1 u^1)^\top \Pi_1 b < -(\Pi_2u^1)^\top \Pi_2 b - (\Pi_2u^1)^\top E\Pi_2 u^2 + (\Pi_3 u^2)^\top \Pi_3 b.
$$
As $E$ is symmetric this means that 
$$
- (u^2)^\top b < - (u^1)^\top b,
$$
which exactly implies the assertion as at equilibrium $V(u) = -\frac{1}{2}u^\top b$.  
\end{proof}

\section*{Acknowledgements}
The authors have no competing interests to declare that are relevant to the content of this article. We thank two anonymous referees for their valuable remarks which largely contributed to the final shape of the manuscript.  The work of JAL, PA, and FST  was partially supported by Proyectos Fondo Europeo de Desarrollo Regional (FEDER) and Consejer\'{\i}a de Econom\'{\i}a, Conocimiento, Empresas y Universidad de la Junta de Andaluc\'{\i}a, by Programa Operativo FEDER 2014-2020 reference P20-00592. The work of JAL was also partially supported by  FEDER Ministerio de Econom\'{\i}a, Industria y Competitividad grant PGC2018-096540-B-I00, and Proyectos Fondo Europeo de Desarrollo Regional (FEDER) and Consejer\'{\i}a de Econom\'{\i}a, Conocimiento, Empresas y Universidad de la Junta de Andaluc\'{\i}a, by Programa Operativo FEDER 2014-2020 reference US-1254251. The work of JAL and PK was supported by  Ministerio de Ciencia e Innovaci\'{o}n of Kingdom of Spain under project No. PID2021-122991NB-C21. The work of PK was also supported by Polish National 
Agency for Academic Exchange (NAWA) under Bekker Programme, project number PPN/BEK/2020/1/00265, as well as by National Science Centre (NCN) of Poland under Projects No. DEC-2017/25/B/ST1/00302 and UMO-2016/22/A/ST1/00077.

\end{document}